\documentclass[amstex,12pt,reqno]{amsart}
\usepackage{amsmath, amssymb, amsthm}
\usepackage{mathrsfs}

\newtheorem{theorem}{Theorem}

\newtheorem{proposition}[theorem]{Proposition}

\theoremstyle{definition}

\newtheorem*{remark}{Remark}

\title{Bi-Lipschitz quasiconformal extensions}
\author[K. Matsuzaki]{Katsuhiko Matsuzaki}
\address{Department of Mathematics, School of Education, Waseda University \endgraf
Shinjuku, Tokyo 169-8050, Japan}
\email{matsuzak@waseda.jp}

\makeatletter
\@namedef{subjclassname@2020}{%
\textup{2020} Mathematics Subject Classification}



\thanks{Research supported by 
Japan Society for the Promotion of Science (KAKENHI 23H01078)}

\begin{document}

\maketitle

\begin{abstract}
We survey several methods for extending quasisymmetric homeomorphisms of the real line to bi-Lipschitz diffeomorphisms of the upper half-plane with respect to the hyperbolic metric.
\end{abstract}

\section{Introduction}

An orientation-preserving homeomorphism $F$ of a domain $\Omega \subset \mathbb{C}$ into the complex plane $\mathbb{C}$ is called \textit{quasiconformal} if $F$ has a locally integrable derivative in the distribution sense and the \textit{complex dilatation} $\mu_F(z) = \bar{\partial}F(z)/\partial F(z)$ 
defined for almost every $z \in \Omega$ satisfies 
that its $L^\infty$-norm $\Vert \mu_F \Vert_\infty$ is less than $1$. For a quasiconformal homeomorphism $F$, the \textit{maximal dilatation} is defined by
$$
K(F)=\frac{1+\Vert \mu_F \Vert_\infty}{1-\Vert \mu_F \Vert_\infty}.
$$

When considering the following maps, we examine various ways of extending them continuously to quasiconformal homeomorphisms of the complex plane $\mathbb{C}$ (or the Riemann sphere $\widehat{\mathbb{C}}$):
\begin{itemize}
\setlength{\parskip}{0cm}
\setlength{\itemsep}{0cm}
\item A quasisymmetric homeomorphism $f: \mathbb{R} \to \mathbb{R}$ on the real line $\mathbb{R}$ (or $f: \mathbb{S} \to \mathbb{S}$ on the unit circle $\mathbb{S}$);
\item A conformal mapping $g: \mathbb{H}^* \to \widehat{\mathbb{C}}$ on the lower half-plane $\mathbb{H}^*$ whose image is a quasidisk;
\item A quasisymmetric embedding $f: \mathbb{R} \to \mathbb{C}$.
\end{itemize}
Here, a quasidisk is considered as the image of a half-plane or disk under a quasiconformal self-homeomorphism $F$ of $\widehat{\mathbb{C}}$, and a \textit{quasisymmetric embedding} $f$ is the restriction of $F$ to 
$\mathbb{R}$ or $\mathbb{S}$. As these quasiconformal extensions are not unique, we must choose an appropriate one for our purpose.

In this survey article, we examine quasiconformal extensions that are bi-Lipschitz diffeo\-morphisms with respect to the hyperbolic metrics of the domain and the range. This map can be precisely defined as follows. 
Let $F:\Omega \to \Omega'$ be a homeomorphism between planar domains $\Omega$ and $\Omega'$ 
with hyperbolic distances $d_{\Omega}$ and $d_{\Omega'}$, respectively. 
Then, $F$ is said to be \textit{bi-Lipschitz} in the hyperbolic metric if there exists a constant $L \geq 1$ such that
$$
\frac{1}{L}\, d_{\Omega}(z_1,z_2) \leq d_{\Omega'}(F(z_1),F(z_2)) \leq L\,d_{\Omega}(z_1,z_2)
$$
holds for any $z_1, z_2 \in \Omega$. The least value of such a constant $L$ 
is referred to as the \textit{bi-Lipschitz constant} of $F$, denoted by $L(F)$. 
An orientation-preserving bi-Lipschitz homeomorphism $F$ is quasiconformal 
with a maximal dilatation satisfying $K(F) \leq L(F)^2$. The infinitesimal form of the above inequalities can be expressed as
\begin{equation*}
\frac{1}{L}\,\rho_\Omega(z)|dz| \leq \rho_{\Omega'}(F(z))|dF| \leq L\,\rho_\Omega(z)|dz|
\end{equation*}
where $\rho_\Omega$ and $\rho_{\Omega'}$ are the hyperbolic densities of $\Omega$ and $\Omega'$, respectively.

We primarily consider two cases: when $\Omega$ is the unit disk $\mathbb{D}$ with $\partial \Omega=\mathbb{S}$, 
and when $\Omega$ is the upper half-plane $\mathbb{H}$ with $\partial \Omega=\mathbb{R}$. 
Both cases can be treated equivalently, but we use $\mathbb{D}$ only in Section \ref{2} and $\mathbb{H}$ in the remaining sections. 
Most of the theorems presented in this paper have already been discussed in the literature. 
In Sections \ref{2} and \ref{3}, we introduce two widely used quasiconformal extensions: 
the Douady--Earle and the Beurling--Ahlfors extensions. Section \ref{4} contains recent results 
by H. Wei and the author regarding the FKP extension, 
which is a variant of the Beurling--Ahlfors extension. In Section \ref{5}, we explore the utility of the bi-Lipschitz property of quasiconformal maps in the context of arguments concerning Teichm\"uller spaces. Theorem \ref{biLip} 
in Section \ref{6} is a slight generalization of a known result concerning the Ahlfors--Weill extension. 
Finally, in Section \ref{7}, we consider the bi-Lipschitz extension 
in the Euclidean metric of bi-Lipschitz embeddings onto chord-arc curves.

\smallskip
\noindent
\textit{Notation.}
The comparable relations $\gtrsim$ and $\lesssim$ indicate inequalities up to uniformly bounded multiplicative constants. When both relations hold, we denote it by $\asymp$.

\section{The Douady--Earle extension}\label{2}

For an orientation-preserving self-homeomorphism
$f:\mathbb S \to \mathbb S$ of the unit circle,
its average observed at the origin $0$ of the unit disk $\mathbb D$ is given by
$$
\xi_f(0)=\frac{1}{2\pi} \int_0^{2\pi} f(e^{i\theta})d\theta.
$$
For any $w \in \mathbb D$, the average of $f$ observed at $w$ is defined as
$$
\xi_f(w)=\frac{1}{2\pi} \int_0^{2\pi} \gamma_w \circ f(e^{i\theta}) d\theta \quad (\in \mathbb D),
$$
where $\gamma_w(z)=(z-w)/(1-\bar w z)$ is a M\"obius transformation of $\mathbb D$ that
sends $w$ to $0$.

According to Douady and Earle \cite{DE},
the \textit{conformally barycenter} of $f$ is a point $w_0 \in \mathbb D$ such that $\xi_f(w_0)=0$.
This point exists uniquely.
Then, the conformally barycentric extension $E(f)$ of $f$ at $0$ is defined as
$E(f)(0)=w_0$.
For any $z \in \mathbb D$, this extension is defined by $E(f)(z)=E(f \circ \gamma_z^{-1})(0)$.
From this definition, we observe that $E(f):\mathbb D \to \mathbb D$ exhibits the following \textit{conformal naturality}:
$$
E(\gamma_1 \circ f \circ \gamma_2)=E(\gamma_1) \circ E(f) \circ E(\gamma_2)
$$
for any M\"obius transformations $\gamma_1$ and $\gamma_2$ of $\mathbb S$.
Moreover, we can see that $E(f)$ is continuously extended to $f$ on $\mathbb S$.

For an orientation-preserving self-homeomorphism
$f:\mathbb S \to \mathbb S$,
let $\Phi_f:\mathbb D \times \mathbb D \to \mathbb C$ be a real-analytic map defined by
$$
\Phi_f(z,w)=\frac{1}{2\pi} \int_0^{2\pi} \frac{f(e^{i\theta})-w}{1-\bar w f(e^{i\theta})} \cdot \frac{1-|z|^2}{|z-e^{i\theta}|^2}\, d\theta.
$$
Then, $\Phi_f(z,w)=0$ represents the implicit function of $w=F(z)$ for the barycentric extension $F=E(f)$.
Under the normalization $F(0)=0$ (or equivalently $\int_0^{2\pi}f(e^{i\theta})d\theta=0$), it is verified that
$$
|\partial F(0)|^2-|\bar \partial F(0)|^2=\frac{|(\Phi_f)_z(0)|^2-|(\Phi_f)_{\bar z}(0)|^2}{|(\Phi_f)_w(0)|^2-|(\Phi_f)_{\bar w}(0)|^2}
$$
is positive. By the conformal naturality, we see that the Jacobian determinant $J_F(z)=|\partial F(z)|^2-|\bar \partial F(z)|^2$ is positive at every $z \in \mathbb D$.
Hence, $F=E(f)$ is a real-analytic self-diffeomorphism of $\mathbb D$ (\cite[Theorem 1]{DE}).

Additionally, we assume  that $f$ extends continuously to a quasiconformal self-homeo\-morphism of $\mathbb D$.
We say that $f$ is quasisymmetric if this condition is satisfied. Under this assumption,
the family of all $\gamma_1 \circ f \circ \gamma_2$ for any M\"obius transformations 
$\gamma_1$ and $\gamma_2$ of $\mathbb S$
such that $E(\gamma_1 \circ f \circ \gamma_2)(0)=0$ is relatively compact.
Based on this property, the following result is proved in \cite[Theorem 2]{DE}.

\begin{theorem}[Douady--Earle]
Let $f:\mathbb S \to \mathbb S$ be a quasisymmetric homeomorphism.
Then, 
the conformally barycentric extension $F=E(f):\mathbb D \to \mathbb D$ of $f$ is a bi-Lipschitz real-analytic diffeomorphism 
in the hyperbolic metric on $\mathbb D$. In particular, $F$ is quasiconformal.
\end{theorem}

To this \textit{Douady--Earle extension} $E(f)$, an intrinsic property of the quasisymmetric homeomorphism $f$ on $\mathbb S$ is not directly reflected.
When the complex dilatation $\mu$ of some quasiconformal extension of $f$ to $\mathbb D$ is known,  
the following claim from Cui \cite[Theorem 1]{Cu} provides some information about the complex dilatation of the Douady--Earle extension
in terms of $\mu$.

\begin{proposition}
Suppose that $f:\mathbb S \to \mathbb S$
extends quasiconformally to $\mathbb D$
with complex dilatation $\mu$. Then, the complex dilatation $\tilde \mu$ of $E(f^{-1})^{-1}$ satisfies
$$
|\tilde\mu(\zeta)|^2 \leq C(1-|\zeta|^2)^2 \int_{\mathbb D} \frac{|\mu(z)|^2}{|1-\bar \zeta z|^4}dxdy
$$
for a constant $C>0$ depending only on $\Vert \mu \Vert_\infty$.
\end{proposition}

\section{The Beurling--Ahlfors extension and its variant}\label{3}

Let $f:\mathbb R \to \mathbb R$ be a \textit{quasisymmetric} homeomorphism. The intrinsic definition of this mapping is that
it is increasing and there exists a constant $M \geq 1$ such that
$$
\frac{1}{M} \leq \frac{f(x+t)-f(x)}{f(x)-f(x-t)} \leq M
$$
for every $x \in \mathbb R$ and for every $t>0$. The boundary extension of
a quasiconformal self-homeomorphism of $\mathbb H$ fixing $\infty$ satisfies this condition. 
This characterization was proved by Beurling and Ahlfors \cite{BA} as well as its converse. The least value of such a constant $M$ is called
a \textit{quasisymmetry constant} of $f$ and denoted by $M(f)$.

It is easy to see that an increasing homeomorphism $f:\mathbb R \to \mathbb R$ is quasisymmetric if and only if
$f$ satisfies the following \textit{doubling property}:
there exists a constant $\widetilde M > 1$ such that
$$
f(x+2t)-f(x-2t) \leq \widetilde M(f(x+t)-f(x-t))
$$
for every $x \in \mathbb R$ and for every $t>0$. 
This property implies, in particular, that
$$
f(2^n)-f(-2^n) \leq \widetilde M^n(f(1)-f(-1)).
$$
From this, we can obverse that $|f(x)|$ is of polynomial growth as $|x| \to \infty$.

The \textit{Beurling--Ahlfors extension} $F$ of $f$ is defined as follows:
for any $z=x+iy \in \mathbb H$, let $F(z)=U(x,y)+iV(x,y)$, where
\begin{align*}
U(x,y)
&=\frac{1}{2y} \int_{-y}^y f(x-t)dt;\\
V(x,y)&=\frac{1}{2y} \int_{-y}^0 f(x-t)dt-\frac{1}{2y} \int_{0}^y f(x-t)dt\\
&=\frac{1}{2y} \int_{0}^y (f(x+t)-f(x-t))dt>0.
\end{align*}

\begin{theorem}[Beurling--Ahlfors]\label{BA}
For any quasisymmetric homeomorphism $f:\mathbb R \to \mathbb R$,
the map $F$ is a quasiconformal $C^1$-diffeomorphism of $\mathbb H$ onto $\mathbb H$. The maximal dilatation 
$K(F)$ can be estimated only by the quasisymmetry constant $M(f)$. Moreover,
$F$ is bi-Lipschitz with respect to the hyperbolic metric on $\mathbb H$. 
The bi-Lipschitz constant
$L(F)$ can also be estimated only by $M(f)$.
\end{theorem}

The proof of this theorem can be found in \cite[Section IV.B]{Ah} and \cite[Section I.5]{Le}.

For a function $\varphi$ on $\mathbb R$ in general and for any $y>0$, let
$$
\varphi_{[y]}(x)=\frac{1}{y} \varphi\left(\frac{x}{y} \right).
$$
Denoting the characteristic function of $E \subset \mathbb R$ by $\chi_E$, we define
$$
\phi(x)=\frac{1}{2} \chi_{[-1,1]}(x), \quad \psi(x)=\frac{1}{2} \chi_{[-1,0]}(x)-\frac{1}{2} \chi_{[0,1]}(x).
$$
Then,
the convolutions of $f$ with these functions yield 
\begin{align*}
(f \ast \phi_{[y]})(x)&=\int_{\mathbb R} f(x-t)\phi_{[y]}(t)dt=U(x,y),\\
(f \ast \psi_{[y]})(x)&=\int_{\mathbb R} f(x-t)\psi_{[y]}(t)dt=V(x,y),
\end{align*}
which give the Beurling--Ahlfors extension $F=U+iV$.

We now consider a variant of the Beurling--Ahlfors extension
by replacing the kernels $\phi$ and $\psi$ with
the Gaussian function and its derivative:
$$
\phi(x)=\frac{1}{\sqrt{\pi}} e^{-x^2}; \quad \psi(x)=\phi'(x)=-2x\phi(x).
$$
We then define
$$
U(x,y)=(f \ast \phi_{[y]})(x), \quad V(x,y)=(f \ast \psi_{[y]})(x), \quad F(z)=U(x,y)+iV(x,y)
$$
as before. Since $|f|$ is of polynomial growth, the integrals of the convolution converge absolutely.
This variant was introduced by
Fefferman, Kenig and Pipher \cite{FKP}, and we refer to it as the \textit{FKP extension}.

\begin{remark}
The function
$\phi_{[\sqrt{y}]}(x)=\frac{1}{\sqrt{\pi y}}e^{-x^2/y}$ satisfies the heat equation
$$
\left(\frac{\partial^2}{4\partial x^2}-\frac{\partial}{\partial y}\right)\phi_{[\sqrt{y}]}(x)=0.
$$
\end{remark}

\begin{theorem}[Fefferman--Kenig--Pipher]
Suppose that a quasisymmetric homeomorphism 
$f:\mathbb R \to \mathbb R$ is locally absolutely continuous. Then, the FKP extension $F$ of $f$
is a quasiconformal real-analytic diffeomorphism possessing the properties of the Beurling--Ahlfors extension 
as stated in Theorem \ref{BA}.
\end{theorem}

\begin{proof}[Outline of Proof] (For detailed proofs, see \cite[Lemma 4.4]{FKP}, \cite[Theorem 3.3]{WM-2}, and \cite[Proposition 3.2]{WM-3}.)
We compute $U_x$ and $V_x$ as follows:
\begin{align*}
U_x&=\frac{\partial}{\partial x}\int_{\mathbb R} f(x-t) \phi_{[y]}(t)dt
=\int_{\mathbb R} f'(x-t) \phi_{[y]}(t)dt=(f' \ast \phi_{[y]})(x);\\
V_x&=\frac{\partial}{\partial x}\int_{\mathbb R} f(x-t) \psi_{[y]}(t)dt
=\int_{\mathbb R} f'(x-t) \psi_{[y]}(t)dt=(f' \ast \psi_{[y]})(x).
\end{align*}
By the doubling property of $f$ (where $f'(x)dx$ is a doubling measure), we have
\begin{equation}\label{basic}
(f' \ast \phi_{[y]})(x) \asymp \frac{1}{2y} \int_{|x-t|<y}f'(t)dt, \quad |(f' \ast \psi_{[y]})(x)| \lesssim \frac{1}{2y} \int_{|x-t|<y}f'(t)dt.
\end{equation}
Furthermore, using the fact that the $y$-derivative corresponds to the second $x$-derivative as 
mentioned in the previous remark, we have
\begin{align*}
U_y&=f \ast \frac{\partial \phi_{[y]}}{\partial y}=f \ast \left(\frac{y}{2}\frac{\partial^2 \phi_{[y]}}{\partial x^2}\right)
=\frac{1}{2}f'\ast \left(y \frac{\partial \phi_{[y]}}{\partial x}\right)=\frac{1}{2}f'\ast \psi_{[y]}=\frac{1}{2}V_x;\\
V_y&=f \ast \frac{\partial \psi_{[y]}}{\partial y}=\frac{1}{2}f'\ast \left(y \frac{\partial \psi_{[y]}}{\partial x}\right)
=U_x+\frac{y^2}{2} f' \ast \frac{\partial^2 \phi_{[y]}}{\partial x^2}.
\end{align*}

From these computations, we observe that $|U_x| \asymp |V_y| \gtrsim |V_x| \asymp |U_y|$. We apply this to
the expression
$$
\frac{1+|\mu_F|^2}{1-|\mu_F|^2}=\frac{|F_z|^2+|F_{\bar z}|^2}{|F_z|^2-|F_{\bar z}|^2}
=\frac{U_x^2+U_y^2+V_x^2+V_y^2}{2(U_xV_y-U_yV_x)}.
$$
We can see that the numerator is comparable to $U_x^2$. 
Moreover, through a more detailed estimate, we can deduce that
the denominator satisfies
$U_xV_y-U_yV_x \gtrsim U_x^2$. Consequently, this fraction is bounded,
and $|\mu_F| \leq k<1$ holds for some constant $k$.
In particular, $F$ is an orientation-preserving local diffeomorphism.
All constants and comparability relations $\asymp$ mentioned above depend only on $M(f)$.

By the property of the heat kernel, we have 
$\lim_{y \to 0}U(x,y)=f(x)$ and $\lim_{y \to 0}V(x,y)=0$. 
Thus, $F(z)$ extends continuously to $f(x)$ on $\mathbb R$.
Moreover, the fact
$\lim_{z \to \infty} F(z)=\infty$ can be seen as follows.
We have that
\begin{equation}\label{V}
V(x,y)=(f \ast \psi_{[y]})(x)=y(f' \ast \phi_{[y]})(x).
\end{equation}
As $y$ tends to $\infty$, \eqref{basic} and \eqref{V} imply that $V(x,y)$ tends to $\infty$.
If $y$ is bounded, then $U(x,y)$ clearly tends to $\infty$ as $|x| \to \infty$.
Combining this with local homeomorphy, a topological argument proves that
$F$ is a global homeomorphism of $\mathbb H$ onto itself.

Regarding the bi-Lipschitz property of $F$, we use the formula
\begin{align*}
|F_z|^2&=\frac{1}{4}(U_x^2+V_x^2+U_y^2+V_y^2)+\frac{1}{2}(U_xV_y-V_xU_y)
\leq \frac{1}{2}(U_x^2+V_x^2+U_y^2+V_y^2).
\end{align*}
From this, we deduce that $|F_z| \asymp U_x=(f'\ast \phi_{[y]})(x)$. Using \eqref{V},
we obtain
$|F_z| \asymp V/y={\rm Im}\,F(z)/{\rm Im}\,z$. Since $F$ is quasiconformal, $|dF|/|dx| \asymp |F_z|$, 
and thus $F$ is
bi-Lipschitz with respect to the hyperbolic metric. The bi-Lipschitz constant $L(f)$ 
also depends only on $M(f)$.
\end{proof}

\section{Application and generalization of FKP extension}\label{4}

The space of Beltrami coefficients $\mu$ on the upper half-plane $\mathbb H$ is denoted by
$$
M(\mathbb H)=\{\mu \in L^\infty(\mathbb H) \mid \Vert \mu \Vert_\infty<1\}.
$$
For $p>1$, we define the space of \textit{$p$-integrable Beltrami coefficients} with respect to the hyperbolic metric on $\mathbb H$ by
$$
M_p(\mathbb H)=\{\mu \in M(\mathbb H) \mid \Vert \mu \Vert_p^p=\int_{\mathbb H} |\mu(z)|^p\frac{dxdy}{(2\,{\rm Im}\,z)^2} <\infty\}.
$$
On the lower half-plane $\mathbb H^*$, the space of Beltrami coefficients and its subspace are defined correspondingly.
For any $\mu \in M_p(\mathbb H)$, it is known that 
the quasiconformal homeomorphism 
$F:\mathbb H \to \mathbb H$ with complex dilatation $\mu$ extends to 
a quasisymmetric homeomorphism $f:\mathbb R \to \mathbb R$ such that
it is locally absolutely continuous and $u=\log f'$ belongs to the real Banach space
$$
B_p(\mathbb R)=\{ u \in L^1_{\rm loc}(\mathbb R) \mid \Vert u \Vert_{B_p}^p=
\int_{\mathbb R}\int_{\mathbb R} 
\frac{|u(s)-u(t)|^p}{|s-t|^2} ds dt<\infty\},
$$
which is taken modulo constant functions.
See \cite[Lemma 3.3]{WM-4} and the references therein.

Conversely, any element of $B_p(\mathbb R)$ induces
a quasisymmetric homeomorphism of $\mathbb R$ that extends quasiconformally to $\mathbb H$
with complex dilatation in $M_p(\mathbb H)$.

\begin{theorem}[\mbox{\cite[Corollary 1.3]{WM-3}}]
For every $u \in B_p(\mathbb R)$ $(p>1)$, let
$f(x)=\int_0^x e^{u(t)}dt$.
Then, $f:\mathbb R \to \mathbb R$ is a quasisymmetric homeomorphism and
its FKP extension $F:\mathbb H \to \mathbb H$ has its complex dilatation 
$\mu_F$ in $M_p(\mathbb H)$.
\end{theorem}

This result can be generalized to the complex-valued setting.
Let $\widetilde B_p(\mathbb R)$ be the complex Banach space of
complex-valued functions containing
$B_p(\mathbb R)$ as the real subspace.
Namely, any element $w \in \widetilde B_p(\mathbb R)$
is represented by $w=u+iv$ for $u, v \in B_p(\mathbb R)$. 
We call $\widetilde B_p(\mathbb R)$ the \textit{$p$-Besov space} on $\mathbb R$.

For any $w \in \widetilde B_p(\mathbb R)$,
we define a map
$\widetilde f^w:\mathbb R \to \mathbb C$ by
$$
\widetilde f^w(x)=\int_0^x e^{w(t)}dt=\int_0^x e^{u(t)}e^{iv(t)}dt.
$$
Furthermore, to define the the complex FKP extension of $\widetilde f^w=\widetilde f$, we set
$$
\widetilde U(x,y)=(\widetilde f \ast \phi_{[|y|]})(x), \quad \widetilde V(x,y)=(\widetilde f \ast \psi_{[|y|]})(x)  
$$
for $x \in \mathbb R$ and $y \neq 0$, and set
\begin{align*}
\widetilde F_{\mathbb H}(z)&=\widetilde U(x,y)+i\widetilde V(x,y) \quad (y>0);\\
\widetilde F_{\mathbb H^*}(z)&=\widetilde U(x,y)-i\widetilde V(x,y) \quad (y<0).
\end{align*}
Then, $\widetilde F^w=\widetilde F$ for
$w \in \widetilde B_p(\mathbb R)$ is given by
$$
\widetilde F(z) = \left\{
\begin{array}{ll}
\widetilde F_{\mathbb H}(z) & (z \in \mathbb H)\\
\widetilde f(z) & (z \in \mathbb R)\\
\widetilde F_{\mathbb H^*}(z) & (z \in \mathbb H^*),
\end{array}
\right. \qquad \qquad
$$
which is not necessarily a quasiconformal homeomorphism of $\mathbb C$ for every $w \in \widetilde B_p(\mathbb R)$.

\begin{theorem}[\mbox{\cite[Theorem 4.5]{WM-3}}]
There exists a neighborhood $N$ of $B_p(\mathbb R)$ in 
$\widetilde B_p(\mathbb R)$ such that
for every $w \in N$, $\widetilde f^w$ is a quasisymmetric embedding and $\widetilde F^w$ is a quasiconformal self-homeomorphism of $\mathbb C$ with
$\mu_{\widetilde F_{\mathbb H}} \in M_p(\mathbb H)$ and $\mu_{\widetilde F_{\mathbb H^*}} \in M_p(\mathbb H^*)$.
Moreover, the correspondence $w \mapsto (\mu_{\widetilde F_{\mathbb H}},\mu_{\widetilde F_{\mathbb H^*}})$ gives
a holomorphic map $\Lambda:N \to M_p(\mathbb H) \times M_p(\mathbb H^*)$.
\end{theorem}

\section{Effects of bi-Lipschitz property}\label{5}
In this section, we illustrate how the bi-Lipschitz property can be utilized in the theory of quasiconformal mappings.

\subsection{Integrability of the complex dilatation of the composition of quasiconformal maps}

For $\mu \in M(\mathbb H)$, let $F^\mu$ denote the quasiconformal self-homeomorphism of $\mathbb H$ with
the normalization fixing $0,1,\infty$.
For any Beltrami coefficients
$\mu, \nu \in M(\mathbb H)$, the complex dilatation of $F^\nu \circ (F^{\nu})^{-1}$ is denoted by
$\mu \ast \nu^{-1}$. Specifically, it is represented as
\begin{equation}\label{composition}
\mu \ast \nu^{-1}(z)=\frac{\mu(\zeta)-\nu(\zeta)}{1-\overline{\nu(\zeta)}\mu(\zeta)}\cdot\frac{\partial F^{\nu}(\zeta)}{\overline{\partial F^{\nu}(\zeta)}} \qquad (z=F^{\nu}(\zeta)).
\end{equation}

For $\mu, \nu \in M_p(\mathbb H)$ $(p \geq 1)$, we compute the norm of $\mu \ast \nu^{-1}$ as follows:
\begin{align*}
\Vert \mu \ast \nu^{-1} \Vert_p^p&=\int_{\mathbb H}|\mu \ast \nu^{-1}(z)|^p \frac{1}{(2\,{\rm Im}\,z)^2}dxdy
=\int_{\mathbb H}
\left|\frac{\mu(\zeta)-\nu(\zeta)}{1-\overline{\nu(\zeta)} \mu(\zeta)}\right|^p \frac{J_{F^\nu}(\zeta)}{(2\,{\rm Im}\,F^{\nu}(\zeta))^2}d\xi d\eta,
\end{align*}
where $J_{F^\nu}$ stands for the Jacobian determinant of $F^\nu$.
Here, if $F^\nu$ is bi-Lipschitz with respect to the hyperbolic metric on $\mathbb H$, then
\begin{align*}
\frac{J_{F^\nu}(\zeta)}{(2\,{\rm Im}\,F^{\nu}(\zeta))^2}&=\frac{|\partial{F^\nu}(\zeta)|^2-|\bar \partial{F^\nu}(\zeta)|^2}{(2\,{\rm Im}\,F^{\nu}(\zeta))^2}
\asymp \frac{|\partial{F^\nu}(\zeta)|^2}{(2\,{\rm Im}\,F^{\nu}(\zeta))^2} \asymp \frac{1}{(2\,{\rm Im}\,\zeta)^2}.
\end{align*}
The comparability $\asymp$ depends only on $K(F^\nu)$ and $L(F^\nu)$.
Thus, there is a constant $C \geq 1$ depending only on $\Vert \nu \Vert_\infty$ and $L(F^\nu)$ such that
\begin{equation}\label{norm}
\frac{1}{C}\Vert \mu \ast \nu^{-1} \Vert_p^p \leq \int_{\mathbb H}|\mu(\zeta)-\nu(\zeta)|^p \frac{d\xi d\eta}{(2\,{\rm Im}\,\zeta)^2}
=\Vert \mu-\nu \Vert_p^p \leq C\Vert \mu \ast \nu^{-1} \Vert_p^p.
\end{equation}
These inequalities establish the relationship between the norm topology on $M_p(\mathbb H)$ 
and the Teichm\"uller topology.

By setting $\mu=0$ in \eqref{norm}, we have $\Vert \nu^{-1} \Vert_p \asymp \Vert \nu \Vert_p$. This implies that if $F^\nu$ is bi-Lipschitz,
then the complex dilatation $\nu^{-1}$ of the inverse $(F^\nu)^{-1}$ belongs to $M_p(\mathbb H)$.

\subsection{Quasiconformal reflection}

Let $\nu \in M(\mathbb H)$ be a complex dilatation of a quasiconformal self-homeomorphism of
$\mathbb H$ that is bi-Lipschitz with respect to the hyperbolic metric on $\mathbb H$.
We consider a quasiconformal self-homeomorphism $F$ of $\mathbb C$ that is conformal on 
$\mathbb H^*$ and has complex dilatation $\nu$ on $\mathbb H$. Then,
$F|_{\mathbb H}:\mathbb H \to F(\mathbb H)$ is a quasiconformal homeomorphism that is bi-Lipschitz with respect to
the hyperbolic metrics on $\mathbb H$ and its image $\Omega=F(\mathbb H)$.

Let $\Omega^*=F(\mathbb H^*)$. The quasiconformal reflection $j:\Omega \to \Omega^*$ with respect to $\partial \Omega$
is defined as
$$
j(z)=F(\overline{F^{-1}(z)}) \quad (z \in \Omega).
$$
This is an orientation-reversing quasiconformal homeomorphism that is bi-Lipschitz in hyperbolic metrics on $\Omega$ and $\Omega^*$.
We denote the hyperbolic densities on 
$\Omega$ and $\Omega^*$ by $\rho_{\Omega}$ and $\rho_{\Omega^*}$, respectively.
Then, the bi-Lipschitz property combined with the anti-quasiconformality implies that there exists a constant $L \geq 1$ such that
\begin{equation}\label{omega1}
\frac{1}{L}\, \rho_{\Omega}(z) \leq |\bar \partial j(z)|\rho_{\Omega^*}(j(z)) \leq L\, \rho_{\Omega}(z)
\end{equation}
for every $z \in \Omega$.

Let us take any $\zeta \in \mathbb H$ and its complex conjugate $\bar \zeta \in \mathbb H^*$.
If $z=F(\zeta)$, then $j(z)=F(\bar \zeta)$. 
We choose the nearest point $w$ in $\partial \Omega$ from $z$ and let $\xi=F^{-1}(w) \in \mathbb R$.
Using the inequality $|\zeta-\xi| \geq |\zeta-\bar \zeta|/2$ and the quasisymmetry of the quasiconformal homeomorphism $F$ on $\mathbb C$ (see \cite[Theorem 3.5.3]{AIM}),
we obtain $|z-w| \gtrsim |z-j(z)|$. Similarly, taking $\xi^*=F^{-1}(w^*) \in \mathbb R$ for the nearest point $w^*$ in $\partial \Omega^*$ from $j(z)$,
we have $|j(z)-w^*| \gtrsim |z-j(z)|$ by $|\bar \zeta-\xi^*| \geq |\zeta-\bar \zeta|/2$ and the quasisymmetry of $F$.
Finally, by using $\rho_\Omega^{-1}(z) \asymp |z-w|$ and $\rho_{\Omega^*}^{-1}(j(z)) \asymp |j(z)-w^*|$, we conclude that
\begin{equation}\label{omega2}
|z-j(z)|^2 \asymp \rho_\Omega^{-1}(z)\rho_{\Omega^*}^{-1}(j(z)).
\end{equation}
The combination of \eqref{omega1} and \eqref{omega2} asserts that
there is a constant $C \geq 1$ with the same dependence as before such that
$$
\frac{1}{C}\, \rho_{\Omega^*}^{-2}(j(z)) \leq
|z-j(z)|^2 |\bar \partial j(z)| \leq C\, \rho_{\Omega^*}^{-2}(j(z)).
$$
This inequality is used for the construction of a local inverse of the Bers projection as the generalized Ahlfors--Weill extension.

Missing proofs and related results can be found in Earle and Nag \cite{EN}, and also in \cite[Section IV.D]{Ah}, \cite[Section II.4]{Le},
and Sugawa \cite{Su}.

\section{An explicit representation of the Ahlfors--Weill extension}\label{6}

A sufficient condition for a conformal mapping $g:\mathbb H^* \to \mathbb C$ 
on the lower half-plane ${\mathbb H}^*$ to be extendable to a quasiconformal homeomorphism of
$\mathbb C$ is given by the \textit{Schwarzian derivative} $S_g$ of $g$. Here,
$$
S_g=(P_g)'-\frac{1}{2}(P_g)^2,
$$
where $P_g=(\log g')'$ is the \textit{pre-Schwarzian derivative}. The hyperbolic $L^\infty$-norms for $P_g$ and $S_g$ are defined as follows:
$$
\Vert P_g \Vert_\infty=\sup_{z \in \mathbb H^*} 2\,|{\rm Im}\,z||P_g(z)|;\quad \Vert S_g \Vert_\infty=\sup_{z \in \mathbb H^*} 4\,|{\rm Im}\,z|^2|S_g(z)|.
$$

\begin{theorem}[Ahlfors--Weill]\label{AW}
If a conformal mapping $g:\mathbb H^* \to \mathbb C$ $(\lim_{z \to \infty}g(z)=\infty)$ satisfies
$\Vert S_g \Vert_\infty<2$, then there is a quasiconformal extension $G$ of $g$ to
$\mathbb C$ such that its complex dilatation on $\mathbb H$ is given by
\begin{equation}\label{harmonic}
\mu_G(z)=-\frac{1}{2}\,|z-\bar z|^2 S_g(\bar z).
\end{equation}
\end{theorem}

In the original proof by Ahlfors and Weill \cite{AhW} (see \cite[Section VI.C]{Ah}), 
the above quasiconformal extension 
was expressed in terms of the linearly independent solutions $\eta_1$ and $\eta_2$ of the ordinary differential equation
$\eta''=-\eta S_g/2$,
and its complex dilatation was computed explicitly.
In the case of $\mathbb D$, this quasiconformal extension itself is represented explicitly in 
Chuaqui and Osgood \cite[p.666]{CO}. Then, applying the Cayley transformation $\mathbb H \to \mathbb D$, we obtain the formula for $G$ on $\mathbb H$:
\begin{equation}\label{formula}
G(z)=g(\bar z)+\frac{(z-\bar z)g'(\bar z)}{1-\frac{1}{2}(z-\bar z)P_g(\bar z)} \quad (z \in \mathbb H).
\end{equation}

Using the theory of Loewner chains, this formula for $G$
can also be obtained.
Indeed, Gumenyuk and Hotta \cite[Proposition 5.5]{GH} considered the chordal Loewner equation on the half-plane and proved that
\begin{equation*}\label{family0}
g_t(z)=g(z-it)-\frac{2tg'(z-it)}{1+tP_g(z-it)} \quad (z \in \mathbb H^*, \ t \geq 0)
\end{equation*}
produces a family of univalent holomorphic functions on $\mathbb H$, which is a chordal Loewner chain
starting from $g=g_0$. By tracing the boundary values of the Loewner chain
$\{g_t\}$ on $\mathbb R$, we obtain the quasiconformal extension. More precisely, 
\eqref{formula} is given by $G(z)=g_y(x)$ for $z=x+iy \in \mathbb H$.
By directly computating the complex dilatation of $G$, we obtain the formula for
$\mu_G$ as in \eqref{harmonic}. 

For a conformal homeomorphism $g$ of $\mathbb H^*$ with $\Vert S_g \Vert_\infty<2$,
the Beltrami coefficient $\mu \in M(\mathbb H)$ given in \eqref{harmonic}
is called \textit{harmonic}.

\begin{theorem}\label{biLip}
For every harmonic Beltrami coefficient $\mu \in M(\mathbb H)$,
the quasiconformal self-homeomorphism $F:\mathbb H \to \mathbb H$ with complex dilatation $\mu$
is a bi-Lipschitz real-analytic diffeo\-morphism with respect to the hyperbolic metric on $\mathbb H$.
\end{theorem}

\begin{proof}
We will prove this theorem by showing that
the quasiconformal extension $G:\mathbb H \to G(\mathbb H)$ is a bi-Lipschitz diffeomorphism
with respect to the hyperbolic metrics under the circumstances of Theorem \ref{AW}.
To this end, we verify that the hyperbolic densities on
$\Omega=G(\mathbb H)$ and $\Omega^*=g(\mathbb H^*)$ satisfies
\begin{equation}\label{Gg1}
\rho_{\Omega}(G(z)) \asymp \rho_{\Omega^*}(g(\bar z)).
\end{equation}
In addition, we also verify that
\begin{equation}\label{Gg2}
|G_z(z)| \asymp |g'(\bar z)|
\end{equation}
for the conformal homeomorphism $g$ and its quasiconformal extension $G$ given in \eqref{formula}.
Then, by
$$
\rho_{\Omega}(G(z))|G_z(z)| \asymp \rho_{\Omega^*}(g(\bar z))|g'(\bar z)| \asymp \frac{1}{2\,|{\rm Im}\,\bar z|}=\frac{1}{2\,{\rm Im}\,z},
$$
we see that $G$ is a bi-Lipschitz real-analytic diffeomorphism. 

By a direct computation in formula \eqref{formula} for $G$, we have
\begin{equation}\label{Gz}
G_z(z)=\frac{g'(\bar z)}{(1-\frac{1}{2}(z-\bar z) P_g(\bar z))^2}.
\end{equation}
Here, the assumption $\Vert S_g \Vert_\infty<2$ implies 
$\Vert P_g \Vert_\infty<2$, which is implicitly shown in the proof of \cite[Proposition 5.5]{GH}.
Hence, the denominator of the fraction in \eqref{Gz} is uniformly bounded away from $0$, and thus \eqref{Gg2} follows.

Finally, we prove \eqref{Gg1}.
Considering $\rho_{\Omega}^{-1}(G(z)) \asymp d(\partial \Omega,G(z))$, we choose a shortest segment $\gamma$
from $G(z)$ to $\partial \Omega$.
The quasiconformality of $G$ implies that
\begin{align*}
d(\partial \Omega,G(z))=\int_{\gamma}|dw| \asymp \int_{G^{-1}(\gamma)}|G_z(z)||dz|.
\end{align*}
Let $\overline{G^{-1}(\gamma)}$ be the reflection of 
the curve $G^{-1}(\gamma)$ with respect to $\mathbb R$. It follows from \eqref{Gg2} that
$$
\int_{G^{-1}(\gamma)}|G_z(z)||dz| \asymp \int_{\overline{G^{-1}(\gamma)}}|g'(\bar z)||d\bar z|,
$$
where the right side term represents the length of the curve $g(\overline{G^{-1}(\gamma)})$.
Since this curve joins $g(\bar z)$ and $\partial \Omega^*$,
the length is bounded from below by $d(\partial \Omega^*,g(\bar z)) \asymp \rho_{\Omega^*}^{-1} (g(\bar z))$.
From these arguments, we obtain
$$
\rho_{\Omega}^{-1}(G(z)) \gtrsim \rho_{\Omega^*}^{-1} (g(\bar z)).
$$
If we begin with considering $\rho_{\Omega^*}^{-1} (g(\bar z))$, we obtain the converse inequality in a similar manner.
Hence, the proof of \eqref{Gg1} is completed.
\end{proof}

\begin{remark}
Takhtajan and Teo \cite[Chap.1, Lemma 2.5]{TT} proved using a different method that the quasiconformal self-homeomorphism $F^\mu$ with
a harmonic Beltrami coefficient $\mu$ is bi-Lipschitz
when $\Vert \mu \Vert_\infty$ is sufficiently small and utilized this property for the theory of
Weil--Petersson Teichm\"uller space.
\end{remark}

In Theorems \ref{AW} and \ref{biLip},
the norm of the Schwarzian derivative of a conformal mapping 
$g:\mathbb H^* \to \mathbb C$ is assumed to be less than $2$.
For an arbitrary conformal mapping $g$ that is quasiconformally extendable to $\mathbb C$,
we can also obtain a bi-Lipschitz extension $G$ by renewing that quasiconformal extension.
An equivalent but more precise formulation of this fact can be stated as follows.
This is essentially the same as the claim given in \cite[Section II.5.2]{Le}.
The similar arguments for certain specific classes of Beltrami coefficients such as $M_p(\mathbb H)$
and their applications
are presented in \cite{WM-1} and \cite{M}.

We say that $\mu_1$ and $\mu_2$ are \textit{Teichm\"uller equivalent} if the normalized quasiconformal self-homeomorphisms $F^{\mu_1}$ and $F^{\mu_2}$ of
$\mathbb H$ have the same boundary extension to $\mathbb R$.

\begin{theorem}
For every $\mu \in M(\mathbb H)$, there exists a finite number of harmonic Beltrami coefficients
$\tilde \mu_1, \ldots, \tilde \mu_n \in M(\mathbb H)$ such that
$$
\tilde \mu=\tilde \mu_n \ast \cdots \ast \tilde \mu_1
$$
is Teichm\"uller equivalent to $\mu$. In particular, the quasiconformal self-homeomorphism $F^{\tilde \mu}$ of $\mathbb H$
is a bi-Lipschitz real-analytic diffeomorphism with respect to the hyperbolic metric on $\mathbb H$. The number $n$ depends only on
$\Vert \mu \Vert_\infty$.
\end{theorem}

\begin{proof}
If $\mu \in M(\mathbb H)$ satisfies $\Vert \mu \Vert_\infty<\frac{1}{3}$,
then a conformal homeomorphism $g$ of $\mathbb H^*$ that extends to $\mathbb H$ quasiconformally
with complex dilatation $\mu$ satisfies $\Vert S_g \Vert_\infty<2$ (see \cite[Theorem II.3.2]{Le}).
We choose $n \in \mathbb N$ such that 
$$
\frac{\Vert \mu \Vert_\infty}{n(1-\Vert \mu \Vert_\infty^2)}<\frac{1}{3}.
$$

Let $\mu_k=k\mu/n$ for $k=0,1,\ldots, n$. We consider $\mu_k \ast \mu_{k-1}^{-1}$ for $k=1,\ldots, n$. 
Formula \eqref{composition} implies that $\Vert \mu_k \ast \mu_{k-1}^{-1} \Vert_\infty<\frac{1}{3}$.
Then, by taking the conformal homeomorphism $g_k$ of $\mathbb H^*$ quasiconformally extendable to $\mathbb H$
with complex dilatation $\mu_k \ast \mu_{k-1}^{-1}$, we have a harmonic Beltrami coefficient $\tilde \mu_k \in M(\mathbb H)$ for $g_k$
by Theorem \ref{AW}. Let $\tilde \mu=\tilde \mu_n \ast \cdots \ast \tilde \mu_1$. Since $\mu_k \ast \mu_{k-1}^{-1}$ is
Teichm\"uller equivalent to $\tilde \mu_k$ for each $k=1,\ldots, n$,
we see that $\mu$ is Teichm\"uller equivalent to $\tilde \mu$.
Since $F^{\tilde \mu_k}$ is a bi-Lipschitz diffeomorphism, the composition $F^{\tilde \mu}=F^{\tilde \mu_n} \circ \cdots \circ F^{\tilde \mu_1}$ is also a bi-Lipschitz diffeomorphism.
\end{proof}

\section{The bi-Lipschitz extension in the Euclidean metric}\label{7}

Let $f:\mathbb R \to \mathbb C$ be a bi-Lipschitz embedding with respect to the Euclidean metric.
In other words, there exists a constant $L \geq 1$ such that 
$$
\frac{1}{L}\,|x-x'| \leq |f(x) -f(x')| \leq L\,|x-x'| 
$$
for any $x, y \in \mathbb R$. The image $\Gamma$ of such an embedding $f$ is called 
an (unbounded) \textit{chord-arc curve}.
It can be observed that $\Gamma$ is locally rectifiable and
its arc length between any two points on $\Gamma$ is bounded by $L^2$ times the distance between them.
Conversely, if a locally rectifiable Jordan curve $\Gamma$ passing through $\infty$ satisfies this property, then its arc length parametrization 
$f:\mathbb R \to \mathbb C$ is a bi-Lipschitz embedding.

Regarding the quasiconformal extension of such a special quasisymmetric embedding, we have the following result.

\begin{theorem}\label{euc}
Every bi-Lipschitz embedding $f:\mathbb R \to \mathbb C$ in the Euclidean metric
extends to a bi-Lipschitz self-homeomorphism $F$ of $\mathbb C$ in the Euclidean metric.
\end{theorem}

\begin{remark}
An orientation-preserving bi-Lipschitz self-homeomorphism $F$ of $\mathbb C$ in the Euclidean metric is quasiconformal
with a maximal dilatation $K(F) \leq L(F)^2$, where $L(F)$ is the bi-Lipschitz constant of $F$.
The bi-Lipschitz property of $F$ as well as that of an embedding $f:\mathbb R \to \mathbb C$
in the Euclidean metric is equivalent to the bi-Lipschitz property in the spherical metric
(see \cite[Section 7]{Tu}).
Indeed, assuming that $F(0)=0$ by an affine translation, we can see that $|F(z)|/|z|$ is uniformly bounded from above and below
in both cases.
Thus, we can extend the bi-Lipschitz property of $F$ to 
the Riemann sphere $\widehat {\mathbb C}$ in the spherical metric by setting $F(\infty)=\infty$.
Moreover, this property is invariant under a M\"obius transformation.
\end{remark}

Additionally, a bi-Lipschitz self-homeomorphism of $\mathbb C$ in the Euclidean metric
satisfies the bi-Lipschitz property on a half-plane or a disk in the hyperbolic metric.

\begin{proposition}
If $F:\mathbb C \to \mathbb C$ is a bi-Lipschitz homeomorphism in the Euclidean metric, then
$F|_{\mathbb H}$ and $F|_{\mathbb H^*}$ are bi-Lipschitz homeomorphisms in the hyperbolic metrics.
\end{proposition}

\begin{proof}
We can consider the bi-Lipschitz property of $F$ in the infinitesimal form
given by
$L^{-1}|dz| \leq |dF| \leq L |dz|$ for a constant $L \geq 1$.
Therefore, for the bi-Lipschitz property of $F|_{\mathbb H}$ with respect to the hyperbolic metrics
on $\mathbb H$ and on $F(\mathbb H)=\Omega$, that is,
$$
\frac{1}{\widetilde L}\, \frac{|dz|}{{\rm Im}\, z} \leq \frac{|dF|}{\rho^{-1}_\Omega(F(z))} \leq \widetilde L\, \frac{|dz|}{{\rm Im}\, z}
$$
for a constant $\widetilde L \geq 1$, it suffices to prove that ${\rm Im}\,z \asymp \rho^{-1}_\Omega(F(z))$. 
By $\rho^{-1}_\Omega(F(z)) \asymp d(F(z),F(\mathbb R))$, this follows from ${\rm Im}\,z \asymp d(F(z),F(\mathbb R))$. 
The same argument holds for $F|_{\mathbb H^*}$.

Let us take the nearest point $\xi$ in $F(\mathbb R)$ from $F(z)$.
The bi-Lipschitz property of $F$ in the Euclidean metric implies that
$$
d(F(z),F(\mathbb R))=d(F(z),\xi) \leq \frac{1}{L}\, d(z,F^{-1}(\xi)) \leq \frac{1}{L}\, {\rm Im}\,z.
$$
The converse inequality is also derived if we begin with taking the nearest point $x$ in $\mathbb R$ from $z$.
\end{proof}

Theorem \ref{euc} was reaffirmed by Tukia \cite{Tu}, in which his previous proof was updated.
Other proofs can also be found in Jerison and Kenig \cite[Proposition 1.13]{JK}, Pommerenke \cite[Theorem 7.10]{Pom}, and Semmes \cite[Lemma 4.11]{Se}.
The arguments were carried out by extending $f$ to a quasiconformal homeomorphism,
and the bi-Lipschitz property of the quasiconformal extension in the hyperbolic metric was utilized.
For example, the Beurling--Ahlfors extension was used in \cite{Tu} and the Douady--Earle extension in \cite{Pom}.
In \cite{Se}, 
a modified Beurling--Ahlfors extension with 
compactly supported smooth kernels $\varphi$ and $\psi$ was consecutively applied to obtain a further property of this extension.

\end{document}